%% file: main.tex
\documentclass{article}
\usepackage[utf8]{inputenc}

\usepackage{amsmath}
\usepackage{amsthm}
\usepackage{amssymb}
\usepackage{upgreek}
\usepackage{xcolor}
\usepackage{comment}
\usepackage{tikz-cd}

\theoremstyle{plain}
\newtheorem{thm}{Theorem}[section]
\newtheorem{lemma}[thm]{Lemma}
\newtheorem{prop}[thm]{Proposition}
\newtheorem{cor}[thm]{Corollary}
\theoremstyle{definition}

\def\Z{\mathbb{Z}}
\def\N{\mathbb{N}}

\providecommand{\abs}[1]{\lvert#1\rvert}

\title{Right Nilpotency of Braces of Cardinality $p^4$}
\author{Dora  Pulji\'{c}\thanks{School of Mathematics, JCMB, The King's Buildings, University of Edinburgh, EH9 3BF, d.puljic@sms.ed.ac.uk} }

\begin{document}

\maketitle
\begin{abstract}

\input{abstract}

\end{abstract}

\input{intro}

\input{background}

\input{thm}

\input{cp2cp2}

\input{vii}

\input{ix}

\input{xixiixiii}

\input{viii}

\input{x}

\input{cpcp3}

\cleardoublepage

\addcontentsline{toc}{chapter}{Bibliography}
\bibliographystyle{acm}
\bibliography{bibliography}

\end{document}

%% file: abstract.tex
We determine right nilpotency of braces of cardinality $p^4$. If a brace of cardinality $p^4$ has an abelian multiplicative group, then it is left and right nilpotent, so we only consider braces with non-abelian multiplicative groups. We show right nilpotency in all cases using the sufficient condition of $A\ast c=0$ for some central element $c$ of a brace $A$.

%% file: intro.tex
\section{Introduction}

A brace is a triple $(A, +,\circ)$ where $(A,+)$ is an abelian group, $(A,\circ)$ is a group and 
\[a\circ(b+c)+a=a\circ b+a\circ c\]
for all $a,b,c\in A$. We  refer to $(A,\circ)$ as the multiplicative group of the brace.

Braces were introduced by W. Rump in 2007 \cite{rump2007braces}, as a generalisation of Jacobson radical rings, in order to help study  solutions of the Yang-Baxter equation. Braces have been studied extensively since - connections to many concepts such as integral group rings \cite{sysak2012adjoint}, Garside groups \cite{chouraqui2010garside}, groups with bijective 1-cocycles \cite{cedo2010involutive,etingof1999set}, quantum groups \cite{etingof1999set,doikou2019braces} and trusses \cite{brzezinski2019trusses} have been found, to name a few. 
In 2016, skew braces were introduced by L. Guarnieri and L. Vendramin in \cite{Guarnieri_2016} as a generalisation of braces in order to study non-involutive solutions to the Yang-Baxter equation. They were further studied in \cite{smoktunowicz2018skew,bachiller2018solutions}, for example.

Advancements in the classification of braces have been made: cyclic braces were classified in \cite{rump2007classification, rump2019classification}, braces of cardinality $pq$ have been classified in \cite{acri2020skew} and of cardinality $p^2q$ in \cite{dietzel2021braces}, skew braces of cardinality $p^3$ have been described in \cite{nejabati2018hopf}, and all not right nilpotent $\mathbb F_{p}$-braces of cardinality $p^{4}$ were described in \cite{puljic2021braces}. 

The importance of right nilpotency in braces comes from their associated set-theoretic solutions having a finite multipermutation level. This class of solutions, and therefore right nilpotent braces, is well understood and investigated \cite{gateva2004combinatorial, gateva2012multipermutation, cedo2014braces, bachiller2017family, cedo2017yang}.

In this paper we determine right nilpotency of braces of cardinality $p^4$ and show that the only not right nilpotent braces of cardinality $p^4$ are the ones constructed in \cite{puljic2021braces}. Braces of cardinality $p^4$ with an abelian multiplicative group are left and right nilpotent \cite{cedo2014braces}, so we will only consider braces with non-abelian multiplicative groups. A result of Bachiller's \cite[Theorem 2.5]{bachiller2016counterexample} states that for a brace of order $p^4$ for $p>5$ with an abelian additive group, the orders of elements in the additive and multiplicative groups coincide. Hence, if the the additive group is $C_{p^4}$, then the multiplicative group contains an element of order $p^4$ and is cyclic as well. Braces with additive group $C_p\times C_p\times C_p\times C_p$ were considered in \cite{puljic2021braces}, so braces with additive groups $C_{p^2}\times C_{p^2}$ and $C_p\times C_{p^3}$ remain to be studied.

%% file: background.tex
\section{Right Nilpotent Braces}

A  \textbf{brace} is a triple $(A, +,\circ)$ where $(A,+)$ is an abelian group, $(A,\circ)$ is a group and 
\[a\circ(b+c)+a=a\circ b+a\circ c\]
for all $a,b,c\in A$. The star operation  $\ast$ is  defined  as \begin{equation*}
\label{star}
    a \circ b = a \ast b + a + b.
\end{equation*}
Then, equivalently, a brace is a triple $(A, +,\ast)$ where $(A,+)$ is an abelian group, $(A,\ast)$ is a group and 
\[a\ast(b+c)=a\ast b+a\ast c\]
for all $a,b,c\in A$. We will refer to $(A,\circ)$ as the multiplicative group of the brace.

A brace is \textbf{left nilpotent} if there exists $n\in\N$ such that $A^{n}=0$, where $A^{i+1}=A\ast A^i$ and $A^1=A$. A brace is \textbf{right nilpotent} if there exists $n\in\N$ such that $A^{(n)}=0$, where $A^{(i+1)}=A^{(i)}\ast A$ and $A^{(1)}=A$. A brace is \textbf{strongly nilpotent} if there exists $n\in\N$ such that $A^{[n]}=0$, where $A^{[i+1]}=\sum_{j=1}^{i}A^{[j]}*A^{[i+1-j]}$ and $A^{[1]}=A$.

For a brace $A$ we will denote the set of elements $c\in A$ such that $c\ast a=a\ast c$ for all $a\in A$ by $Z(A)$, and we will call such elements central.
We will use the following throughout: 
By \cite{puljic2021braces} if a central element $c$ of a brace $A$ is such that $A\ast c=0$, then $c\ast A=0$.  Therefore $c$ generates an ideal $I$ in $A$ such that $I\ast A = A\ast I=0$. By a result of Bachiller's \cite{bachiller2016counterexample}, this implies that $A/I$ is right nilpotent as it is a brace of cardinality $p^3$ or less. Hence $A$ is right nilpotent. Therefore, to show the brace $A$ is right nilpotent it suffices to show that $A\ast c=0$ for some $c\in Z(A)$.

%% file: thm.tex
We begin with a theorem.

\begin{thm}\label{thm}
Let $A$ be a brace of cardinality $p^4$ with $(A,+)$ either $C_p\times C_{p^3}$ or $ C_{p^2}\times C_{p^2}$. If $(A,+) \cong C_p\times C_{p^3}$, then we let $m=2$, and if $(A,+) \cong  C_{p^2}\times C_{p^2}$, then we let $m=1$.  Suppose $(A,\circ)$ is generated by $\{P,Q_1,Q_2,\dots, Q_i \}$ for some $i$. We assume the following:
\begin{enumerate}
    \item $P^{p^m}$ is central,
    \item  $P^{p^m}\in A^2$,
    \item $ord(Q_i)\leq p^m$,
    \item we can write any $a\in A$ as $P^{k}\circ\prod_{\circ}Q_j^{a_j}$, where $k,a_i\in \Z$ and the $Q_j$ in the product appear in any order.
\end{enumerate}
Then $P\ast p^m P^k =0$ for all $k\in \Z$.
\end{thm}

We will apply this theorem in the following sections to specific cases of braces to show that we can find a central element, namely $P^{p^m}$, for which $A\ast P^{p^m}=0$. The goal of this section is to prove this theorem. We now assume the assumptions of the theorem.

\begin{prop}\label{PPn}
$P\ast P^n = \frac{1}{6}(n-2)(n-1)n P\ast (P\ast (P\ast P))+ \frac{1}{2}n(n-1) P\ast (P\ast P)+ n (P\ast P).$
\end{prop}
\begin{proof}
We prove the proposition by induction. The base case clearly holds. Now, we have
\begin{align*}
    P\ast P^{n+1} &= P\ast (P\circ P^n) = P\ast (P\ast P^n)+P\ast P +P\ast P^n\\
    &= \frac{1}{2}n(n-1)P\ast( P\ast (P\ast P))+ n P\ast(P\ast P)+P\ast P\\
    &\;\;\;+\frac{1}{6}(n-2)(n-1)n P\ast (P\ast (P\ast P))+ \frac{1}{2}n(n-1) P\ast (P\ast P) \\
    &\;\;\;+n( P\ast P)\\
    &=\frac{1}{6}(n-1)n(n+1) P\ast (P\ast (P\ast P))+ \frac{1}{2}n(n+1) P\ast (P\ast P)\\
    &\;\;\;+(n+1) P\ast P.
\end{align*}
\end{proof}

\begin{prop}\label{prop1}
$p^m(P\ast P)\in A^3$ and $P\ast P^{p^m} = \frac{1}{2}p^m(p^m-1) P\ast (P\ast P)+ p^m (P\ast P).$
\end{prop}

\begin{proof}
By Proposition \ref{PPn} we have
\begin{align*}
    P\ast P^{p^m} &=\frac{1}{6}(p^m-2)(p^m-1)p^m P\ast (P\ast (P\ast P))+ \frac{1}{2}p^m(p^m-1) P\ast (P\ast P)\\
    &\;\;\;+ p^m (P\ast P), 
\end{align*}
which implies that $p^m(P\ast P)\in A^3. $ This in turn implies that 
\[P\ast P^{p^m} = \frac{1}{2}p^m(p^m-1) P\ast (P\ast P)+ p^m (P\ast P). \]
\end{proof}

\begin{cor}\label{cor1}
$p^m P\ast(P\ast P) = P\ast(P\ast P^{p^m}).$
\end{cor}

\begin{prop}\label{prop2}
$p^m P\ast(P\ast P^n)=np^m P\ast (P\ast P)=0$ for all $n\in\Z$ and $P^{p^m} = p^m P$.
\end{prop}
\begin{proof}
As before we have
\begin{align*}
    P\ast P^{n} &=\frac{1}{6}(n-2)(n-1)n P\ast (P\ast (P\ast P))+ \frac{1}{2}n(n-1) P\ast (P\ast P)\\
    &\;\;\;+ n (P\ast P), 
\end{align*}
which implies
\[p^m P\ast(P\ast P^n)=np^m P\ast (P\ast P).\]
Note that $P^{p^m}\in p^m A.$ Let $a\in A$ so that $P^{p^m}=p^m a$. Then we can write $a=P^{k} \circ \prod_{\circ}Q_i^{a_i}$, where  $k,a_i\in \Z$ and  the $Q_i$ in the product appear in some order. Then we have
\begin{align*}
    a &= P^{k}\ast (\prod_{\circ}Q_i^{a_i})+P^{k}+ \prod_{\circ}Q_i^{a_i}\\
    &=P^{k}\ast (\sum \prod_{\ast}Q_i^{a_i})+P^{k}+ \sum \prod_{\ast}Q_i^{a_i},
\end{align*}
where $\sum \prod_{\ast}Q_i^{a_i}$ signifies some sum of $\ast$-products of $Q_i$. Hence
\[P^{p^m}=p^m a = p^m P^{k}.\]
We have
\begin{align*}
    p^m P\ast (P\ast P) = p^m P\ast (P\ast P^{k}) = p^m k P\ast(P\ast P).
\end{align*}
This implies $k=1$ and $P^{p^m} = p^m P$.We have
\[P\ast P^{p^m} =  \frac{1}{2}p^m(p^m-1) P\ast (P\ast P)+ p^m (P\ast P),\]
so
\[\frac{1}{2}p^m(p^m-1) P\ast (P\ast P)=0.\]
As $\frac{1}{2}(p^m-1)$ is an integer which is not divisible by $p$, it follows that $p^mP\ast (P\ast P)=0.$
\end{proof}

\begin{cor}\label{np2pp}
$p^m(P\ast P^n) = np^m(P\ast P).$
\end{cor}

\begin{prop}
$p^mP^{-1}=-p^m P.$
\end{prop}

\begin{proof}

Recall that $P^{p^m}, P^{-p^m}\in p^m A$, so we can write $P^{p^m}=p^m b$, $P^{-p^m}=p^m c$ for some $b,c\in A$. Let $b=P^{b_1} \circ \prod_{\circ}Q_i^{b_i}$, $c=P^{c_1} \circ \prod_{\circ}Q_i^{c_i}$ for some $b_i,c_i\in \Z$, so \[P^{p^m}=p^m P^{b_1}, \quad P^{-p^m}=p^m P^{c_1}.\] 
Now consider $P^{-p^m}\ast P^{p^m}$. We have
\begin{align*}
    P^{-p^m}\ast P^{p^m} &= -P^{-p^m}-P^{p^m}\\
    &=-p^mP^{c_1}-p^mP^{b_1}.
\end{align*}
Also, as $P^{kp^m}\in Z(A)$
\begin{align*}
    P^{-p^m}\ast P^{p^m} &= p^m(P^{-p^m}\ast P^{b_1})\\
    &=p^m(P^{b_1}\ast p^mP^{c_1})\\
    &=0,
\end{align*}
so $p^m P^{b_1}=-p^mP^{c_1}$. Therefore
\begin{align*}
    P^{-p^m}\ast P^{p^m} &=-(p^mP^{b_1}\ast p^mP^{b_1})\\
    &=-(p^m P^{b_1})^2+2p^mP^{b_1}\\
    &=0,
\end{align*}
so $(p^m P^{b_1})^2=2p^mP^{b_1}$, which implies $(p^m P^{b_1})^n=np^mP^{b_1}$. Hence $(P^{p^m})^n=nP^{p^m}$ and similarly $(P^{-p^m})^n=nP^{-p^m}$. It follows that $P\ast P^{-p^m} = P\ast -P^{p^m}=-p^m(P\ast P).$ Now we have
\begin{align*}
    P\ast P^{-p^m} &= P\ast (-p^m P^{b_1})\\
    &=-p^m(P\ast P^{b_1})\\
    &=-p^m b_1 (P\ast P)
\end{align*}
by Corollary \ref{np2pp}. Hence $b_1=1.$ Now we have
\begin{align*}
    P\ast P^{-p^m} &= P\ast (p^m P^{c_1})\\
    &=p^m(P\ast P^{c_1})\\
    &=p^m c_1 (P\ast P),
\end{align*}
so $c_1=-1.$

\end{proof}

\begin{lemma}
Let $a\in A$. Then $P\ast p^m a = 0.$
\end{lemma}

\begin{proof}
Let $a\in A$. Then we can write $a=P^{n} \circ \prod_{\circ}Q_i^{a_i}$, where  $n,a_i\in \Z$ and  the $Q_i$ in the product appear in some order. Then, as before,
\[p^m a = p^m P^{n}.\]
We have
\[P\ast p^m P^n = np^m P\ast P.\]
Consider $p^m(P\ast P^{p^m-1}).$ We have
\begin{align*}
    p^m(P\ast P^{p^m-1}) &= -p^m (P\ast P)\\
    &= p^m(P\ast (P^{p^m}\circ P^{-1}))\\
    &= p^m(P\ast (P^{-1}\ast P^{p^m})+P\ast P^{-1}+P\ast P^{p^m})\\
    &=-p^m P-p^mP^{-1}\\
    &=0.
\end{align*}
\end{proof}

This lemma finishes the proof of Theorem \ref{thm}.

Given a brace $(A,+,\circ)$, we define a map $\lambda$ as
\[ \lambda_{a}(b)=a*b+b=a\circ b-a  \]
for $a,b\in A$.
A known property of $\lambda$ is that
\[\lambda _{a\circ b}(c)=\lambda _{a}\left(\lambda _{b}(c)\right).\]
From this property it follows that commutators, i.e. elements of the form $a^{-1}\circ b^{-1}\circ a\circ b$ for $a,b\in A$, are in $A^2$. 

Theorem \ref{thm} requires the element $P^{p^m}$ to be in $A^2$. We will achieve this condition in the braces we will consider by choosing $P^{p^m}$ that is a commutator.

%% file: cp2cp2.tex
\section{$C_{p^2}\times C_{p^2}$}
We suppose $(A,+)\cong C_{p^2}\times C_{p^2}$. As aforementioned, the  multiplicative orders of elements equal the additive. Following the classification of groups of order $p^4$ in \cite{burnside1911theory}, the non-abelian groups with elements of order $p^2$, but no elements of order $p^3$ are generated by elements in $\{P,Q,R\}$ such that

\begin{itemize}
    \item \textbf{(VII)} $P^{p^2}=Q^p=R^p=1, \;\; PQ=QP,\;\;  PR=RP,\;\; R^{-1}QR=QP^p$ or
    \item \textbf{(VIII)} $P^{p^2}=Q^{p^2}=1, \;\; Q^{-1}PQ=P^{1+p}$ or
    \item \textbf{(IX)} $P^{p^2}=Q^p=R^p=1, \;\; PQ=QP, QR=RQ,\; \; R^{-1}PR=P^{1+p} $ or
    \item \textbf{(X)} $P^{p^2}=Q^p=R^p=1, \;\; R^{-1}PR=PQ,\;\; PQ=QP,\;\; QR=RQ$ or
    \item $P^{p^2}=Q^p=R^p=1, \;\; Q^{-1}PQ=P^{1+p},\;\;  R^{-1}PR=PQ,\;\; R^{-1}QR=P^{\alpha p}Q,$ where
    \begin{itemize}
        \item \textbf{(XI)} $\alpha = 0$ or
        \item \textbf{(XII)} $\alpha =1$ or
        \item \textbf{(XIII)} $\alpha =$ any non-residue $\mod p$.
    \end{itemize}
\end{itemize}

We will refer to the different groups by the roman numerals corresponding to them above. The groups above are the options for the multiplicative group $(A,\circ)$. 

%% file: vii.tex
\subsection{VII}
Suppose $(A,\circ)$ is the group VII given by the relations
 \[P^{p^2}=Q^p=R^p=1, \;\; P\circ Q=Q\circ P,\;\;  P\circ R=R\circ P,\;\; R^{-1}\circ Q\circ R=Q\circ P^p.\]
 Then $P\in Z(A)$, which implies that $P^k\in Z(A)$ for all $k\in \Z$. In particular, $P^p\in Z(A)$ and $P^p\in pA.$ Also, notice that as $P^p$ is a commutator, it is in $A^2$.

Notice now that $P^{p}\in p A$, so to show $A$ is right nilpotent it suffices to prove that $A\ast p A = 0$. As any element of $A$ is a product of $P$s, $Q$s and $R$s, it suffices to prove $Q\ast p A=0$, $R\ast pA = 0$ and $P\ast p A = 0.$

 \begin{prop}
 $Q\ast pA = R\ast pA =0.$
 \end{prop}
\begin{proof}
Let $a\in A$. Then we can write $a = P^{a_1}\circ R^{a_2}\circ Q^{a_3}$ for some $a_i\in \Z$. Then we have $p a = pP^{a_1}$. It follows that
\begin{align*}
    Q\ast pa & = Q\ast p P^{a_1}\\
    &= P^{a_1}\ast pQ\\
    &=0,
\end{align*}
and 
\begin{align*}
    R\ast pa & = R\ast p P^{a_1}\\
    &= R^{a_1}\ast pQ\\
    &=0.
\end{align*}
\end{proof}

\begin{cor}
$P\ast p A = 0.$
\end{cor}
\begin{proof}
Follows from Theorem \ref{thm}.
\end{proof}

%% file: ix.tex
\subsection{IX}
Suppose $(A,\circ)$ is the group IX given by the relations
 \[P^{p^2}=Q^p=R^p=1, \;\; P\circ Q=Q\circ P,\;\;  Q\circ R=R\circ Q,\;\; R^{-1}\circ P\circ R= P^{1+p}.\]
 Then $Q\in Z(A)$. We have
 \[R^{-1}\circ P^p\circ R = P^p,\]
 so $P^p\in Z(A)$. Now notice that as
 \[P\circ R = R\circ P^{1+p}, \]
 it follows that
 \[R\circ P^k = P^{k(1-p)}\circ R,\]
 for any $k\in \Z$. Also, notice that as $P^p$ is a commutator, it is in $A^2$.

 Notice now that $P^{p}\in p A$, so to show $A$ is right nilpotent it suffices to prove that $A\ast p A = 0$. As any element of $A$ is a product of $P$s, $R$s and $Q$s, it suffices to prove $Q\ast p A=0$, $R\ast pA=0$ and $P\ast p A = 0.$
 
 \begin{prop}
 $Q\ast pA = R\ast pA =0.$
 \end{prop}
\begin{proof}
Let $a\in A$. Then we can write $a = P^{a_1}\circ R^{a_2}\circ Q^{a_3}$ for some $a_i\in \Z$. Then we have $p a = pP^{a_1}$. It follows that
\begin{align*}
    Q\ast pa & = Q\ast p P^{a_1}\\
    &= P^{a_1}\ast pQ\\
    &=0,
\end{align*}
and 
\begin{align*}
    R\ast pa & = R\ast p P^{a_1}\\
    &= R^{a_1}\ast pQ\\
    &=0.
\end{align*}
\end{proof}

\begin{cor}
$P\ast p A = 0.$
\end{cor}
\begin{proof}
Follows from Theorem \ref{thm}.
\end{proof}

%% file: xixiixiii.tex
\subsection{XI, XII, XIII}
Suppose $(A,\circ)$ is the group given by the relations

\begin{align*}
    P^{p^2}=Q^p=R^p=1, \;\; &Q^{-1}\circ P\circ Q= P^{1+p},\;\;  R^{-1}\circ P\circ R=P\circ Q,\;\; \\
    &R^{-1}\circ Q\circ R= P^{\alpha p}\circ Q,
\end{align*} 
where $\alpha = 0$ in the group XI, $\alpha =1$ in the group XII and $\alpha$ is a non-residue $\mod p$ in the group XIII.
 We have 
 \begin{align*}
      P^n\circ Q &= P^{n-1}\circ P\circ Q\\
      &= P^{n-1}\circ Q\circ P^{1+p}\\
      &=\dots\\
      &=Q\circ P^{(1+p)n}.
 \end{align*}
 Hence
 \begin{align*}
     R^{-1}\circ P^p \circ R &= (P\circ Q)^p = P\circ Q\circ \dots \circ P\circ Q\\
     &=Q\circ P^{1+p}\circ Q\circ P^{1+p}\circ \dots \circ Q\circ P^{1+p}\\
     &=Q^2\circ P^{(1+p)^2+1+p}\circ Q\circ \dots \circ Q\circ P^{1+p}\\
     &=Q^3\circ P^{(1+p)^3+(1+p)^2+1+p}\circ Q\circ \dots \circ Q\circ P^{1+p}\\
     &=\dots\\
     &=Q^p\circ P^{(1+p)^p+(1+p)^{p-1}+\dots +1+p}.
 \end{align*}
 We have
 \begin{align*}
     P^{(1+p)^p+(1+p)^{p-1}+\dots +1+p} &= P^{1+p^2+1+(p-1)p+\dots +1+p}\\
     &=P^{p+p((p-1)+(p-2)+\dots +1)}\\
     &=P^{p+\frac{1}{2}p^2(p-1)}\\
     &=P^p.
 \end{align*}
 Hence
 \[R^{-1}\circ P^p \circ R = Q^p \circ P^p = P^p.\]
 Therefore $P^p\in Z(A).$ Also, notice that as $P^p$ is a commutator, it is in $A^2$.

 Notice now that $P^{p}\in p A$, so to show $A$ is right nilpotent it suffices to prove that $A\ast p A = 0$. As any element of $A$ is a product of $P$s, $R$s and $Q$s, it suffices to prove $Q\ast p A=0$, $R\ast pA=0$ and $P\ast p A = 0.$

 \begin{prop}
 $Q\ast pA = R\ast pA =0.$
 \end{prop}
 
 \begin{proof}
 Let $a\in A$. Then we can write $a = P^{k}\circ \prod_\circ Q_i^{a_i}$ for some $k,a_i\in \Z$, where $Q_j\in \{Q,R\}$ are in some order. Notice that we can write this since we have the relations
 \[R\circ P = P\circ R\circ Q^{-1},\quad Q\circ P = P^{1-p}\circ Q.\]
 Then we have $p a = pP^{k}$. It follows that
\begin{align*}
    Q\ast pa & = Q\ast p P^{k}\\
    &= P^{k}\ast pQ\\
    &=0,
\end{align*}
and 
\begin{align*}
    R\ast pa & = R\ast p P^{k}\\
    &= P^{k}\ast pR\\
    &=0.
\end{align*}
 \end{proof}

\begin{cor}
$P\ast p A = 0.$
\end{cor}
\begin{proof}
Follows from Theorem \ref{thm}.
\end{proof}

%% file: viii.tex
\subsection{VIII}
Suppose $(A,\circ)$ is the group VIII given by the relations
 \[P^{p^2}=Q^{p^2}=1, \;\; Q^{-1}\circ P\circ Q= P^{1+p}.\]
 Then we have
 \[P^p=Q^{-1}\circ P^p\circ Q,\]
 so $P^p\in Z(A).$

\begin{prop}
$Q^p\in Z(A).$
\end{prop}

\begin{proof}
From the group relations we have
\begin{align*}
    P\circ Q^{p}\circ P^{-1}&=(Q\circ P^p)^p\\
    &=Q^p\circ P^{p^2}\\
    &=Q^p,
\end{align*}
where we used the fact that $P^p\in Z(A).$ 
\end{proof}

\begin{prop}\label{nilp}
$A$ is right nilpotent.
\end{prop}

\begin{proof}

We let $S:=\langle Q^p, P^p\rangle$, the group generated by $Q^p$ and $P^p$ under the $\circ$ operation. Hence $S\subseteq Z(A)$ is an abelian group of order $p^2$. Now, as $Q^p,P^p\in pA$, we have $S\subseteq pA$.
As $(A,+)\cong C_{p^2}\times C_{p^2}$, $\abs{pA}=p^2$ and $S=pA\subseteq Z(A)$.

If  $A\ast pA=0$, then we are done.

Now suppose $A\ast pA\neq 0.$ Notice that $pA\neq A\ast pA$, as in that case we would have
\[A\ast pA = A\ast (A\ast (A\ast (A\ast pA)))=0.\]
It follows that $\abs{A\ast pA}=p$, hence $A\ast pA\cong C_p$. Similarly, note that $ A\ast (A\ast pA)\neq A\ast pA$, so $\abs{A\ast (A\ast pA)}=1$. Hence for any element $c\in A\ast pA$, we have $A\ast c=0$. Hence $A$ is right nilpotent.

\end{proof}

%% file: x.tex
\subsection{X}
Suppose $(A,\circ)$ is the group X given by the relations
 \[P^{p^2}=Q^p=R^p=1, \;\; P\circ Q=Q\circ P,\;\;  Q\circ R=R\circ Q,\;\; R^{-1}\circ P\circ R= P\circ Q.\]
 Then $Q\in Z(A)$. We have
 \[R^{-1}\circ P^p\circ R = (Q\circ P)^p = Q^p\circ P^p = P^p,\]
 so $P^p\in Z(A)$.
 
\begin{prop}
$A$ is right nilpotent.
\end{prop}

\begin{proof}

We let $S:=\langle Q, P^p\rangle$, the group generated by $Q$ and $P^p$ under the $\circ$ operation. Hence $S\subseteq Z(A)$ is an abelian group of order $p^2$.

Now, as $Q,P^p\in pA$, we have $S\subseteq pA$.
As $(A,+)\cong C_{p^2}\times C_{p^2}$, $\abs{pA}=p^2$ and $S=pA\subseteq Z(A)$.

We now proceed as in Proposition \ref{nilp}.

\end{proof}

%% file: cpcp3.tex
\section{$C_p\times C_{p^3}$}

In this section we characterise the right nilpotency of braces of cardinality $p^4$ with additive group $C_p\times C_{p^3}$.

Suppose the brace $(A,+,\circ)$ has $(A,+)\cong C_p\times C_{p^3}$. This implies that $(A,\circ)$ has an element of order $p^3.$ Now, by \cite{burnside1911theory} the only non-abelian group of order $p^4$ with an element of order $p^3$ is defined by the relations
\[P^{p^3}=1,\;\; Q^p=1,\;\; Q^{-1}\circ P\circ Q = P^{1+p^2}.\]


Notice that $P^{p^2}\in Z(A)$. Now, as $P^{p^2}$ is a commutator we have $P^{p^2}\in A^2 $.

 Notice now that $P^{p^2}\in p^2 A$, so to show $A$ is right nilpotent it suffices to prove that $A\ast p^2 A = 0$. As any element of $A$ is a product of $P$s and $Q$s, it suffices to prove $Q\ast p^2 A=0$, and $P\ast p^2 A = 0.$

\begin{prop}
$Q\ast p^2 A = P\ast p^2 A=0.$
\end{prop}
\begin{proof}

Let $a\in A$. Then we can write $a = P^{a_1}\circ Q^{a_2}$ for some $a_i\in \Z$. We have
\begin{align*}
    p^2a &= p^2(P^{a_1}\ast Q^{a_2} + P^{a_1}+Q^{a_2})\\
    &=p^2P^{a_1}.
\end{align*}
By Theorem \ref{thm} we have $P\ast p^2 P^{a_1}=0$. We have
\begin{align*}
    Q\ast p^2a &= Q\ast p^2P^{a_1}\\
    &=p^2(P^{a_1-a_1p^2}\ast Q -P^{a_1}+P^{a_1-a_1p^2})\\
    &=0.
\end{align*}

\end{proof}

\begin{center}
\subsubsection*{\sc{Acknowledgements}} 
\end{center}
The author is thankful to Agata Smoktunowicz for her invaluable comments and corrections.

%% file: main.bbl
\begin{thebibliography}{10}

\bibitem{acri2020skew}
{\sc Acri, E., and Bonatto, M.}
\newblock Skew braces of size {$pq$}.
\newblock {\em Comm. Algebra 48}, 5 (2020), 1872--1881.

\bibitem{bachiller2016counterexample}
{\sc Bachiller, D.}
\newblock Counterexample to a conjecture about braces.
\newblock {\em J. Algebra 453\/} (2016), 160--176.

\bibitem{bachiller2018solutions}
{\sc Bachiller, D.}
\newblock Solutions of the {Y}ang-{B}axter equation associated to skew left
  braces, with applications to racks.
\newblock {\em J. Knot Theory Ramifications 27}, 8 (2018), 1850055, 36.

\bibitem{bachiller2017family}
{\sc Bachiller, D., Ced\'{o}, F., Jespers, E., and Okni\'{n}ski, J.}
\newblock A family of irretractable square-free solutions of the
  {Y}ang-{B}axter equation.
\newblock {\em Forum Math. 29}, 6 (2017), 1291--1306.

\bibitem{brzezinski2019trusses}
{\sc Brzezi\'{n}ski, T.}
\newblock Trusses: between braces and rings.
\newblock {\em Trans. Amer. Math. Soc. 372}, 6 (2019), 4149--4176.

\bibitem{burnside1911theory}
{\sc Burnside, W.}
\newblock {\em Theory of groups of finite order}.
\newblock Dover Publications, Inc., New York, 1955.
\newblock 2d ed.

\bibitem{cedo2017yang}
{\sc Ced\'{o}, F., Gateva-Ivanova, T., and Smoktunowicz, A.}
\newblock On the {Y}ang-{B}axter equation and left nilpotent left braces.
\newblock {\em J. Pure Appl. Algebra 221}, 4 (2017), 751--756.

\bibitem{cedo2010involutive}
{\sc Ced\'{o}, F., Jespers, E., and del R\'{\i}o, A.}
\newblock Involutive {Y}ang-{B}axter groups.
\newblock {\em Trans. Amer. Math. Soc. 362}, 5 (2010), 2541--2558.

\bibitem{cedo2014braces}
{\sc Ced\'{o}, F., Jespers, E., and Okni\'{n}ski, J.}
\newblock Braces and the {Y}ang-{B}axter equation.
\newblock {\em Comm. Math. Phys. 327}, 1 (2014), 101--116.

\bibitem{chouraqui2010garside}
{\sc Chouraqui, F.}
\newblock Garside groups and {Y}ang-{B}axter equation.
\newblock {\em Comm. Algebra 38}, 12 (2010), 4441--4460.

\bibitem{dietzel2021braces}
{\sc Dietzel, C.}
\newblock Braces of order {$p^2q$}.
\newblock {\em J. Algebra Appl. 20}, 8 (2021), Paper No. 2150140, 24.

\bibitem{doikou2019braces}
{\sc Doikou, A., and Smoktunowicz, A.}
\newblock From braces to hecke algebras \& quantum groups.
\newblock {\em arXiv preprint arXiv:1912.03091\/} (2019).

\bibitem{etingof1999set}
{\sc Etingof, P., Schedler, T., and Soloviev, A.}
\newblock Set-theoretical solutions to the quantum {Y}ang-{B}axter equation.
\newblock {\em Duke Math. J. 100}, 2 (1999), 169--209.

\bibitem{gateva2004combinatorial}
{\sc Gateva-Ivanova, T.}
\newblock A combinatorial approach to the set-theoretic solutions of the
  {Y}ang-{B}axter equation.
\newblock {\em J. Math. Phys. 45}, 10 (2004), 3828--3858.

\bibitem{gateva2012multipermutation}
{\sc Gateva-Ivanova, T., and Cameron, P.}
\newblock Multipermutation solutions of the {Y}ang-{B}axter equation.
\newblock {\em Comm. Math. Phys. 309}, 3 (2012), 583--621.

\bibitem{Guarnieri_2016}
{\sc Guarnieri, L., and Vendramin, L.}
\newblock Skew braces and the {Y}ang-{B}axter equation.
\newblock {\em Math. Comp. 86}, 307 (2017), 2519--2534.

\bibitem{nejabati2018hopf}
{\sc Nejabati~Zenouz, K.}
\newblock On hopf-galois structures and skew braces of order p\^{} 3.

\bibitem{puljic2021braces}
{\sc Pulji\'{c}, D., Smoktunowicz, A., and Nejabati~Zenouz, K.}
\newblock Some braces of cardinality {$p^4$} and related {H}opf-{G}alois
  extensions.
\newblock {\em New York J. Math. 28\/} (2022), 494--522.

\bibitem{rump2007braces}
{\sc Rump, W.}
\newblock Braces, radical rings, and the quantum {Y}ang-{B}axter equation.
\newblock {\em J. Algebra 307}, 1 (2007), 153--170.

\bibitem{rump2007classification}
{\sc Rump, W.}
\newblock Classification of cyclic braces.
\newblock {\em J. Pure Appl. Algebra 209}, 3 (2007), 671--685.

\bibitem{rump2019classification}
{\sc Rump, W.}
\newblock Classification of cyclic braces, {II}.
\newblock {\em Trans. Amer. Math. Soc. 372}, 1 (2019), 305--328.

\bibitem{smoktunowicz2018skew}
{\sc Smoktunowicz, A., and Vendramin, L.}
\newblock On skew braces (with an appendix by {N}. {B}yott and {L}.
  {V}endramin).
\newblock {\em J. Comb. Algebra 2}, 1 (2018), 47--86.

\bibitem{sysak2012adjoint}
{\sc Sysak, Y.~P.}
\newblock The adjoint group of radical rings and related questions.
\newblock In {\em Ischia group theory 2010}. World Sci. Publ., Hackensack, NJ,
  2012, pp.~344--365.

\end{thebibliography}
